\documentclass{amsart}

\usepackage{amsmath}

\usepackage{amsfonts}

\usepackage{hyperref}

\usepackage{latexsym}

\usepackage{amssymb}

\usepackage{mathrsfs}

\usepackage{geometry}
\numberwithin{equation}{section}

\usepackage[utf8]{inputenc} 

\usepackage{amscd}

\newtheorem{theorem*}{Theorem}
\newtheorem{definition*}{Definition}
\newtheorem{remark*}{Remark}

\newtheorem{theorem}{Theorem}[section]
\newtheorem{lemma}[theorem]{Lemma}
\newtheorem{corollary}[theorem]{Corollary}
\newtheorem{proposition}[theorem]{Proposition}
\newtheorem{definition}[theorem]{Definition}

\newtheorem{remark}[theorem]{Remark}

\newcommand{\ma}[1]{\ensuremath{\mathbb{#1}}}
 
\font\bb=msbm7 at 10 pt

\def \Z {\hbox{\bb Z}}
\def \Q {\hbox{\bb Q}}
\def \N {\hbox{\bb N}}

\def \F {\hbox{\bb F}}

\def \Tr {\mbox{\rm{Tr}}}

\def \Q {\hbox{\bb Q}}
\def \I {\mbox{\bf I}}

\def \AS {\mathscr{AS}}

\newcommand{\NP}{\ensuremath{\mbox{\rm{NP}}}}
\newcommand{\GNP}{\ensuremath{\mbox{\rm{GNP}}}}
\newcommand{\Ima}{\ensuremath{\mbox{\rm{Im }}}}

\newcommand{\Ker}{\ensuremath{\mbox{\rm{Ker}}}}

\author{R\'egis Blache}
\address{\'Equipe LAMIA,
\'ESP\'E de Guadeloupe}
\email{rblache@espe-guadeloupe.fr}

\title[First vertices]{Valuations of exponential sums and Artin-Schreier curves}

\begin{document}

\begin{abstract}
Let $p$ denote an odd prime. In this paper, we are concerned with the $p$-divisibility of additive exponential sums associated to one variable polynomials over a finite field of characteristic $p$, and with (the very close question of) determining the Newton polygons of some families of Artin-Schreier curves, i.e. $p$-cyclic coverings of the projective line in characteristic $p$. 

We first give a lower bound on the $p$-divisibility of exponential sums associated to polynomials of fixed degree. Then we show that an Artin-Schreier curve defined over a finite field of characteristic $p$ cannot be supersingular when its genus $g$ has the form $(p-1)\left(i(p^n-1)-1\right)/2$ for some $1\leq i\leq p-1$ and $n\geq 1$ such that $n(p-1)>2$.
We also determine the first vertex of the generic Newton polygon of the family of $p$-rank $0$ Artin-Schreier curves of fixed genus, and the associated Hasse polynomial.
\end{abstract}

\subjclass[2000]{}

\maketitle

\section*{Introduction}

Let $p$ denote an odd prime, and $k=\F_q$, $q=p^m$ a finite field of characteristic $p$. We fix once and for all a non trivial additive character $\psi$ of $k$. For any one variable polynomial $f\in k[x]$, we define the exponential sum
$$S(f):=\sum_{x\in k} \psi(f(x))$$
This is an algebraic integer in $\Z[\zeta_p]$, where $\zeta_p$ denotes a fixed primitive $p$-th root of unity. Since we are interested in the $p$-adic valuation of this sum, we will consider it as
an element in $\Z_p[\zeta_p]$. The $q$-adic valuation over $\Z_p$ normalized by $v_q(q)=1$ extends to this ring in a unique way, and we look for a lower bound for the number $v_q(S(f))$ when $f$ varies among degree $d$ polynomials having their coefficients in $k$.

The one dimensional case of \cite[Theorem 1.2]{as} gives the lower bound $v_q(S(f))\geq \frac{1}{d}$. It is known that this bound is tight when $p\equiv 1 \mod d$ \cite{ro}. Actually it is very close to be tight when the characteristic is large; when we have $p>2 d$, a tight bound is given by $v_q(S(f))\geq \frac{1}{p-1}\lceil \frac{p-1}{d}\rceil$, see \cite[Theorem 1.1]{sz2}.

When the degree is large compared to the characteristic, Moreno and Moreno \cite{mm} have shown the lower bound $v_q(S(f))\geq \frac{1}{\sigma_p(f)}$, where $\sigma_p(f)$ denotes the maximum of the $p$-weights (sums of base $p$ digits) of the exponents effectively appearing in $f$. Note that these last bounds depend on the characteristic $p$.

In this paper, we show the following bounds (note that one can always assume, by Artin-Schreier reduction \cite[Exemple 3.5]{sga}, that the degree of $f$ is prime to $p$)
\begin{theorem*}
\label{bounds}
Let $k$ denote a finite field of odd characteristic $p$, and $f\in k[x]$ be a polynomial of degree $d\geq\frac{p-1}{2}$. We have
$$ v_q(S(f))\geq \left\{
\begin{array}{rcl}
 \frac{1}{n(p-1)} & \textrm{if} & p^n-1\leq d \leq p^{n+1}-p-1\\
 \frac{2}{(2n+1)(p-1)} & \textrm{if} & p^{n+1}-p+1\leq d \leq p^{n+1}-2 \textrm{ and } n\geq 1 \\
 \frac{2}{p-1} & \textrm{if} & \frac{p-1}{2}\leq d \leq p-2,~p\geq 5 \\
\end{array}
\right.$$
moreover these bounds are tight, in the sense that there exists a finite field $k$ of characteristic $p$ and a polynomial $f$ of degree $d$ having its coefficients in $k$ such that the inequalities above are equalities.
\end{theorem*}

Let us consider Artin-Schreier curves, i.e. $p$-cyclic coverings of the projective line in characteristic $p$. In \cite{pz}, the authors study a moduli space $\AS_g$ for genus $g$ such curves, and its $p$-rank stratification by strata $\AS_{g,s}$, $0\leq s\leq g$. In particular, they show that the image of $\AS_g$ under the Torelli morphism is not in general position with respect to the $p$-rank stratification of the space of principally polarized abelian varieties. In this paper, we shall concentrate on the $p$-rank $0$ stratum $\AS_{g,0}$, and on the finer stratification by the Newton polygons of (the numerator of the zeta function of) $p$-rank $0$ Artin-Schreier curves.

From Grothendieck's specialization theorem, these Newton polygons have a lower bound, the \emph{generic Newton polygon}. From \cite{dens}, the above lower bounds on the valuations can be reinterpreted as the first slopes of these generic polygons. We go further and determine their first vertices, with the help of the congruence given in \cite{cong}. We also give the space of curves whose Newton polygon share this vertex; in the space of degree $d$ polynomials parametrized by their coefficients, it is the Zariski open subset given by the non vanishing of some polynomial, the \emph{Hasse polynomial}. This polynomial was determined in the case $p>2d$ in \cite[Theorem 1.1]{sz2} as $\left\{f^{\lceil \frac{p-1}{d}\rceil}\right\}_{p-1}$, the degree $p-1$ coefficient of $f^{\lceil \frac{p-1}{d}\rceil}$. As a consequence, we give a complete answer to this question.

Recall that a genus $g$ \emph{supersingular} curve is a curve whose Newton polygon is the highest possible, i.e. the segment between $(0,0)$ and $(2g,g)$. Many results are available on supersingular Artin-Schreier curves in characteristic $2$; note in this case Artin-Schreier curves are hyperelliptic curves. In \cite{vv}, the authors construct large families of supersingular such curves, then they use them in \cite{vv2} to show that there exists supersingular curves of any genus in characteristic $2$. In \cite{sz}, Scholten and Zhu show a result in the other direction; when $g$ has the form $2^n-1$, $n\geq 2$, there does not exist any supersingular Artin-Schreier curve of genus $g$. As a consequence of our results, we are able to show the following generalization of this last result to odd characteristic

\begin{theorem*}
\label{supsing}
Let $p$ denote an odd prime, and $d=i(p^n-1)$, where $n\geq 1$ and $1\leq i\leq p-1$ are integers such that $n(p-1)>2$. Then there is no supersingular Artin-Schreier curve in characteristic $p$ having genus $g=\frac{1}{2}(p-1)(d-1)$.
\end{theorem*}

Let us briefly describe the methods employed in this paper. In \cite{mkcs}, Moreno, Kumar, Castro and Shum give a link between the valuations of exponential sums over $k$ and the $p$-weights of the solutions of certain systems of equations modulo $q-1$. 

Let us be more precise (in the one variable case): let $D$ denote a finite set of integers; we consider the modular equation
$$\sum_D du_d\equiv 0 \mod q-1,~ \sum_D du_d>0.$$
If $U=(u_d)_{d\in D}$ is a solution, we denote by $s_p(U):=\sum_D s_p(u_d)$ its $p$-weight, and by $\sigma_{D,p}(m)$ the minimum of the $s_p(U)$ over the solutions. Then the main result in \cite{mkcs} is the following: for any polynomial $f$ having its coefficients in $k$, and its exponents in $D$, we have the lower bound $v_q(S(f))\geq \frac{\sigma_{D,p}(m)}{m(p-1)}$, which is an equality for at least one $f$ as above.

In \cite{dens}, we have shown that the infimum 
$$\inf_{m\geq 1} \left\{ \frac{\sigma_{D,p}(m)}{m(p-1)}\right\}$$
is actually a minimum $\delta_{D,p}$, the \emph{$p$-density} of the set $D$. From this result, we deduce the lower bound $v_q(S(f))\geq \delta_{D,p}$ depending only on the set $D$ and the characteristic $p$. The main goal of this paper is to determine the $p$-density of any set $D$ of the form $\{1,\ldots,d\}$. 

\begin{definition*}
For any integer $d$ prime to $p$, we denote by $\delta(p,d)$ the $p$-density $\delta_{D,p}$ of the set $D=\{1,\ldots,d\}$.
\end{definition*}

In this setting, Theorem \ref{bounds} above can be rewritten

$$\delta(p,d)=\left\{
\begin{array}{rcl}
 \frac{1}{n(p-1)} & \textrm{if} & p^n-1\leq d \leq p^{n+1}-p-1\\
 \frac{2}{(2n+1)(p-1)} & \textrm{if} & p^{n+1}-p+1\leq d \leq p^{n+1}-2, \\
 \frac{2}{p-1} & \textrm{if} & \frac{p-1}{2}\leq d \leq p-2,~p\geq 5 \\
\end{array}
\right.$$

Actually the main part of this paper is devoted to show these equalities: they result from Propositions  \ref{cas1}, \ref{wgeom} and \ref{densp2}. Note that the tightness of Moreno et al's bound joint with the definition of the density as a minimum are sufficient to show the last assertion of Theorem \ref{bounds}.

\begin{remark*}
Let us come back to the lower bounds $\sigma_{D,p}(m)$ depending on the field $k$. It follows from the above definitions that we have the inequalities $\sigma_{D,p}(m)\geq \lceil m(p-1)\delta_{D,p}\rceil$ and $v_q(S(f))\geq \frac{1}{m(p-1)}\lceil m(p-1)\delta_{D,p}\rceil$, the last one being valid for any degree $d$ polynomial over $k$. 

One can show that the first one is indeed an equality by constructing explicitely a solution having $p$-weight $\lceil m(p-1)\delta_{D,p}\rceil$. As a consequence there always exists a degree $d$ polynomial such that the second is also an equality, and we have obtained a tight lower bound for the valuations of exponential sums associated to a fixed degree polynomial over a fixed finite field.
\end{remark*}

In \cite{dens}, we reinterpret the density as the first slope of the generic Newton polygon. But in order to conclude about the non existence of supersingular curves and the first vertex, we have to be more precise. We use the main result from \cite{cong}; it gives an explicit congruence for the $L$-function (and as a consequence for the numerator of the zeta function) ``along its first slope'' in terms of certain invariants associated to $D$ and $p$. Since we determine these invariants in the course of giving the density, we can write the congruence explicitely, from which we deduce Theorem \ref{supsing}, the first vertices and their Hasse polynomial.

\medskip

The paper is organized as follows: in section \ref{sec1}, we recall some properties of the solutions of the modular equation; we pay special attention to their supports, in order to give upper bounds for the lengths of minimal irreducible solutions. Once this has been done, we show Theorem \ref{bounds} in the rather technical, but completely elementary Section \ref{sec3}. Finally, we give the consequences of the results in section \ref{sec3} for Artin-Schreier curves in the last section: we first prove Theorem \ref{supsing} in subsection \ref{prove}, then we dedicate subsection \ref{hasse} to the determination of the first vertex and Hasse polynomial of the generic Newton polygons.

\section{Bounds for the support} 
\label{sec1}

In this section, we recall the modular equation from \cite{mkcs}, and some objects and results from \cite{dens}; then we define support maps and give bounds that will be useful in the study of minimal irreducible solutions of the modular equation in the next sections.

\subsection{Solutions of the modular equation}

In the following, we fix a finite set $D\subset \N_{>0}$ and a prime $p$. For any $\ell\geq 1$, we define the finite set 
$E_{D,p}(\ell)\subset \{0,\ldots,p^\ell-1\}^{|D|}$ as the set of solutions $U=(u_d)_{d\in D}$ of the following system
$$\left\{\begin{array}{rcl}
\sum_D du_d & \equiv & 0 \mod p^\ell-1 \\
\sum_D du_d & > & 0 \\
\end{array}
\right.$$

We denote by $s_p(n)$ the $p$-weight of the integer $n$, i.e. the sum of its base $p$ digits. We define the \emph{weight} of a solution as 
$s_p(U):=\sum_D s_p(u_d)$, its \emph{length} as $\ell(U):=\ell$, and its \emph{density} as $\delta(U):=\frac{s_p(U)}{(p-1)\ell(U)}$.

We set $\sigma_{D,p}(\ell):=\min\{s_p(U),~U\in E_{D,p}(\ell)\}$. In \cite{dens}, we have shown that the infimum
$$\inf_{\ell\geq 1} \left\{ \frac{\sigma_{D,p}(\ell)}{\ell(p-1)}\right\}$$
is actually a minimum $\delta_{D,p}$, the \emph{$p$-density} of the set $D$.

\begin{definition}
A solution $U\in E_{D,p}(\ell)$ is \emph{minimal} when we have $\delta(U)=\delta_{D,p}$.
\end{definition}

We define the \emph{shift} as the map $\delta$ from $\{0,\ldots,p^\ell-1\}$ to itself sending $p^\ell-1$ to itself, and any other $i$ to the 
remainder of $pi$ modulo $p^\ell-1$ (note that this map shifts the base $p$ digits, and its inverse is sometimes called the \emph{Dwork map}). We extend it coordinatewise to the set 
$\{0,\ldots,p^\ell-1\}^{|D|}$; then it leaves the subset $E_{D,p}(\ell)$ stable. As a consequence, all integers $\sum_D d\delta^k(u_d)$, $0\leq k\leq \ell-1$, are positive multiples of $p^\ell-1$. 

\begin{definition}
The \emph{support} of the solution $U$ is the map $\varphi_U$ from $\Z/\ell\Z$ to $\N_{>0}$ 
defined by 
$$\varphi_U(k):=\frac{1}{p^\ell -1} \sum_D d\delta^k(u_d)$$
A solution $U$ is \emph{irreducible} when the map $\varphi_U$ is an injection.
\end{definition}

For any $d\in D$ we write the base $p$ expansion $u_d=\sum_{r=0}^{\ell-1} p^ru_{dr}$; note that we have $s_p(U)=\sum_D\sum_{r=0}^{\ell-1}u_{dr}$. Recall from \cite[Lemma 1.2 (ii)]{dens} that for any 
$0\leq r\leq \ell-1$, we have the equalities

\begin{equation}
\label{saut}
\sum_D du_{dr}= p\varphi_U(\ell-r-1)-\varphi_U(\ell-r)
\end{equation}

For any $0\leq r\leq \ell-1$, we define the \emph{$r$-th weight} of the solution $U$ as the integer $w_r:=\sum_{D} u_{dr}$. We list below two easy consequences of these definitions for further use

\begin{lemma}
\label{jump}
Let $U\in E_{D,p}(\ell)$ denote a solution of the above system, with weight $w$ and support $\varphi_U$

\begin{itemize}
	\item[(i)] we have $w=\sum_{r=0}^{\ell-1} w_r$;
	\item[(ii)] for any $0\leq r\leq \ell-1$, we have $p\varphi_U(\ell-r-1)-\varphi_U(\ell-r)\leq w_r\max D$.
\end{itemize}

\end{lemma}

\subsection{Support maps}

From the above results, the supports of irreducible solutions $U$ share many common features: they are periodic, and consist of some geometric sequences of common ratio $p$. From these constraints, we now define, and study, a certain type of maps

\begin{definition}
Let $\ell\geq s$ denote two integers, and $\varphi : \Z/\ell\Z\rightarrow \N_{>0}$ any map 

\begin{itemize}
	\item[(i)] We say that $\varphi$ is a \emph{support map of length $\ell$ with $s$ jumps} if we have $\varphi(i+1)=p\varphi(i)$ except for exactly 
$s$ pairwise distinct values $i_1,\ldots,i_s\in \Z/\ell\Z$, for which we have $\varphi(i+1)<p\varphi(i)$.
\item[(ii)] We call the $s$ positive integers $j_t:=p\varphi(i_t)-\varphi(i_t+1)$ the \emph{jumps} of $\varphi$.
\item[(iii)] Moreover, we say that $\varphi$ is \emph{irreducible} when $\varphi$ is an injection.
\end{itemize}

\end{definition}

We begin with a technical result about such a support map.

\begin{lemma}
\label{max1}
Let $\varphi$ denote a support map of length $\ell$; assume that its maximal jump is at most $M$. Then for any $i\in \Z/\ell\Z$, 
we have the inequality
$$\varphi(i)\leq \frac{M}{p-1}$$
\end{lemma}

\begin{proof}
From the definition of the jumps, we have the inequality $p\varphi(i)-\varphi(i+1)\leq M$ for any $i\in \Z/\ell\Z$. Assume we 
have $\varphi(i_0)> \frac{M}{p-1}$ for some $i_0\in \Z/\ell\Z$; we get $\varphi(i_0+1)\geq p\varphi(i_0)-M>\varphi(i_0)$. Continuing this process, 
we get $\varphi(i_0+\ell)>\varphi(i_0)$, which contradicts the definition of the map $\varphi$.
\end{proof}

Our next problem is to give a lower bound for the sum $|\varphi|:=\sum_i \varphi(i)$, where $\varphi$ is an irreducible support map of length $\ell$ with $s$ jumps.

As a consequence of the definition, we can write
$$\Ima \varphi=\{n_1p^{u_{1}},\ldots,n_1p^{u_{1}+\ell_{1}-1},\ldots,n_sp^{u_{s}},\ldots,n_sp^{u_{s}+\ell_{s}-1}\}$$
where $\sum \ell_i=\ell$, the $\ell$ integers above are pairwise distinct and $(n_i,p)=1$ for all $1\leq i\leq s$.

In order to minimize $|\varphi|$, we can assume $u_1=\ldots=u_s=0$; up to reordering if necessary, we also assume $n_1<\ldots<n_s$. 
We begin with a lemma.

\begin{lemma}
\label{min1}
Let $n_1<\ldots<n_s$ be prime to $p$ integers; denote by $\{b_1,\ldots,b_s\}$ the (ordered) set 
$E_s:=\{1\leq i\leq s+\lceil\frac{s}{p-1}-1\rceil,~(i,p)=1\}$. Then we have $n_i\geq b_i$ for any $1\leq i\leq s$.
\end{lemma} 

\begin{proof}
We just have to show that $E_s$ contains exactly $s$ elements: since these are the first $s$ prime to $p$ integers, the Lemma follows 
immediately from the assumptions on the $n_i$. 

For any positive integer $n$, the set $\{1,\ldots,n\}$ contains exactly $n-\lfloor\frac{n}{p}\rfloor$ prime to $p$ integers; note that 
we can assume $n$ coprime to $p$ since in the opposite case, the sets $\{1,\ldots,n\}$ and $\{1,\ldots,n-1\}$ contain the same number of 
prime to $p$ integers. 

We solve the equation $s=n-\lfloor\frac{n}{p}\rfloor$. If we write the Euclidean division of $n$ by $p$, $n=qp+r$ with $0<r\leq p-1$, we 
get $n-\lfloor\frac{n}{p}\rfloor=(p-1)q+r$. We consider two cases
\begin{itemize}
	\item if we have $1\leq r<p-1$, then $s=(p-1)q+r$ is exactly the Euclidean division of $s$ by $p-1$, and we have 
	$n=s+q=s+\lfloor\frac{s}{p-1}\rfloor$;
	\item else we have $r=p-1$, $s=(p-1)(q+1)$, and $n=p(q+1)-1=s+\lfloor\frac{s}{p-1}\rfloor-1$.
\end{itemize}
The result follows from these equalities.
\end{proof}

We now consider the sequence $(c_i)_{i\geq 1}$ which is defined by ordering the elements in the set 
$$F_s=\bigcup_{b\in E_s} \left\{p^jb,~j\geq 0\right\}$$

\begin{lemma}
\label{min2}
Let $(c_i)_{i\geq 1}$ be as above; then we have
\begin{itemize}
	\item $c_i=i$ for any $1\leq i\leq s+\lceil\frac{s}{p-1}-1\rceil$;
	\item $c_i=pc_{i-s}$ for any $i\geq s+\lceil\frac{s}{p-1}\rceil$.
\end{itemize}
\end{lemma}

\begin{proof}
First consider any $i\leq s+\lceil\frac{s}{p-1}-1\rceil$; then we can write $i=i_0p^{j_0}$ for some prime to $p$ integer $i_0$ with 
$i_0\leq i\leq s+\lceil\frac{s}{p-1}-1\rceil$. From Lemma \ref{min1} above, $i_0$ is in $E_s$, and $i$ in $F_s$ from the definition of this last set; thus the first assertion is true.

Now assume $i\geq s+\lceil\frac{s}{p-1}\rceil$; first note that $c_i\geq i\geq s+\lceil\frac{s}{p-1}\rceil>b_s$, and $c_i$ can be written 
$c_i=pc_k$. We set $t:=\lceil\frac{s}{p-1}\rceil$; we will show inductively that the equality $c_{t+n+s}=pc_{t+n}$ is true for any $n$. 
First note that $c_{t+s}$ is a multiple of $p$ from above, and that it must be the least multiple of $p$ greater than or equal to 
$\lceil\frac{s}{p-1}\rceil+s=t+s$. Thus we must have $c_{t+s}=p  \lceil\frac{s}{p-1}\rceil=pc_t$; we have shown the equality for $n=0$. 
Assume the equality $c_{t+n+s}=pc_{t+n}$ is true for some $n$; note that from our construction, the integer $pc_{n+t+1}$ is an element of 
the sequence $(c_i)$. From this observation and the induction hypothesis, we must have $pc_{t+n}<c_{t+n+s+1}\leq pc_{t+n+1}$. Since 
$c_{t+n+s+1}$ is a multiple of $p$, we have $\frac{1}{p}c_{t+n+s+1}=c_k$ for some $t+n<k\leq t+n+1$. We must have $k=t+n+1$, and this is 
the result
\end{proof}

We now give a lower bound for each element of the sequence $(c_n)$

\begin{lemma}
\label{min3}
Let $n\geq 1$ denote an integer, and set $n=qs+r$, $1\leq r\leq s$; then we have $c_n\geq p^qr$.
\end{lemma}

\begin{proof}
First note that the division above is not the Euclidean one. From the second assertion of the preceding Lemma, it is sufficient to show the 
result for any $1\leq n\leq s+\lceil\frac{s}{p-1}-1\rceil$. The assertion is trivial for any $1\leq n\leq s$ since in this case $c_n=n=r$ and $q=0$. If $n>s$, we write $n=s+i$ for some 
$1\leq i\leq \lceil\frac{s}{p-1}-1\rceil$, we have
$$n=(p-1)\frac{s}{p-1}+i\geq (p-1)\lceil\frac{s}{p-1}-1\rceil+i\geq (p-1)i+i=pi$$
and this is the desired result since in this case we have $q=1$ and $r=i$.
 \end{proof}

We are ready to give a lower bound for the total weights of certain support maps with the help of the results above

\begin{proposition}
\label{bou}
Let $\varphi$ denote an irreducible support map of length $\ell$ with $s$ jumps. Write $\ell=qs+r$ with $1\leq r\leq s$; then we have the 
inequality
$$|\varphi|\geq \sum_{i=1}^\ell c_i \geq \frac{s(s+1)}{2}\frac{p^q-1}{p-1}+\frac{r(r+1)}{2}p^q$$
\end{proposition}

\begin{proof}
Recall that we have written 
$$\Ima \varphi=\{n_1p^{u_{1}},\ldots,n_1p^{u_{1}+\ell_{1}-1},\ldots,n_sp^{u_{s}},\ldots,n_sp^{u_{s}+\ell_{s}-1}\}$$ 
As a consequence of Lemma \ref{min1}, we have $|\varphi|\geq \sum b_i\sum_{j=0}^{\ell_i-1} p^j$. From the equality $\sum_{i=1}^s \ell_i=\ell$, 
we know that the integers $b_ip^j$ are $\ell$ elements in $F_s$: they are the $c_k$, $k\in K$ for some subset $K\subset \N_{>0}$ with cardinality $\ell$. 
As a consequence, we have $\sum_{i=1}^s b_i\sum_{j=0}^{\ell_i-1} p^j=\sum_K c_k\geq \sum_{i=1}^\ell c_i$. The assertion is now an easy 
consequence of Lemma \ref{min3}
$$|\varphi|\geq \sum_{i=1}^\ell c_i \geq \sum_{u=0}^{q-1}\sum_{v=1}^s c_{su+v}+\sum_{t=1}^r c_{qs+t}\geq \sum_{u=0}^{q-1}\sum_{v=1}^s p^uv+
\sum_{t=1}^r p^qt$$

\end{proof}

We now show that the above bound remains valid for support maps having less than $s$ jumps.

\begin{lemma}
\label{min4}
Let $\varphi$ denote an irreducible support map of length $\ell$ with $t$ jumps, with $t\leq s$. Then the bound in Proposition \ref{bou} 
remains valid.
\end{lemma}

\begin{proof}
First consider the sequence $(d_k)$ obtained by ordering the set $F_{t}$; we show the inequality $d_k\geq c_k$ for any $k$. It is sufficient 
to prove this inequality for $t=s-1$.

Note that for any $1\leq k\leq s+\lceil\frac{s}{p-1}-1\rceil$ we have $c_k=k$; since the $d_k$ are pairwise distinct, we get the inequality 
in this case. Now for any $k\geq s+\lceil\frac{s}{p-1}\rceil$, we have $d_k=pd_{k-s+1}\geq pc_{k-s+1}=c_{k+1}>c_k$ and we get the result 
inductively.

As in the proof of Proposition \ref{bou}, we have an inequality $|\varphi|\geq \sum_K d_k\geq \sum_{i=1}^\ell d_i\geq \sum_{i=1}^\ell c_i$, 
the last inequality coming from the beginning of the proof.
\end{proof}

%
%
%

\subsection{Properties of the support}

In this section, we make the link between the preceding subsections, and we show a result that will be useful when we determine the density.

\begin{proposition}
\label{gapgeom}
Let $U$ be a solution of the system of modular equations associated to $D$ and $p$, with weight $w$ and length $\ell$. We have the following
\begin{itemize}
	\item[(i)] the support $\varphi_U$ is a support map of length $\ell$, with at most $w$ jumps; moreover it is irreducible if, and only if
	the solution $U$ is;
	\item[(ii)] if the support $\varphi_U$ has $s$ jumps, then we have the following inequality
	$$  \max \varphi_U\leq  \frac{w-s+1}{p-1}\max D $$
	\item[(iii)] if it has exactly $w$ jumps, then all are elements of $D$. Moreover, the solution $U$ is completely determined by its support in 
	this case.
\end{itemize}

\end{proposition}

\begin{proof}
The first assertion comes from the properties of the map $\varphi_U$; from Lemma \ref{jump}, we have $p\varphi_U(i)\neq \varphi_U(i+1)$ if 
and only if $w_{\ell-i-1}>0$. Since we have $w=\sum w_r$, we get at most $w$ positive elements among the $w_r$, and this is the result.

Assertion (ii) is an easy consequence of Lemma \ref{jump} (ii), Lemma \ref{max1}, and the fact that we must have $\max w_i\leq w-s+1$.

In order to show (iii), first note that we must have $\ell\geq w$ in order for the support to have $w$ jumps. Denote by $i_1,\ldots,i_w$ the $w$
different jumps. From the proof of the first assertion, we must have $w_{\ell-i_k-1}>0$ for $1\leq k\leq w$. Since $w=\sum w_r$, we get $w_{\ell-i_k-1}=1$ for $1\leq k\leq w$, and all other $w_i$ vanish. Now we have $w_r=\sum_D u_{dr}$ for any $0\leq r\leq \ell-1$, thus for any $1\leq k\leq w$, there exists exactly one $d_k\in D$ such that $u_{d_k\ell-i_k-1}=1$ and all others $u_{di}$ are zero. From (\ref{saut}), we get $p\varphi_U(\ell-i_k-1)-\varphi_U(\ell-i_k)=d_k\in D$.

For the last assertion, just remark that with the above notations, for any $d\in D$, we have $u_d=\sum p^{i}$, where the sum is over those 
$i\in\{1,\ldots,\ell\}$ such that $p\varphi_U(\ell-1-i)-\varphi_U(\ell-i)=d$.
\end{proof}

\section{The density, and minimal irreducible solutions}
\label{sec3}

In this section, we fix an odd prime number $p$; we give the $p$-densities of the sets 
$D:=\{1\leq i\leq d,~(i,p)=1\}$. By Artin-Schreier reduction \cite[Exemple 3.5]{sga}, this is sufficient to prove Theorem \ref{bounds}. We also determine minimal irreducible solutions for the modular equations, in order to prove the results about Artin-Schreier curves in the next section.

\subsection{The case \texorpdfstring{$d=p^{n+1}-2$}{pdeux}}

We consider the set $D:=\{1\leq i\leq p^{n+1}-2,~(i,p)=1\}$. We have $\sigma_p(D)=s_p(p^{n+1}-2)=n(p-1)+p-2$; from \cite[Corollary 1.1]{dens}, 
we have the inequality $\delta(p,p^{n+1}-2)\geq \frac{1}{n(p-1)+p-2}$. On the other hand, for any $2\leq i\leq p-1$, the solutions 
$p^n(p^{n+1}-i)+ip^n-1$ have weight $2$ and length $2n+1$; as a consequence, they have density $\frac{2}{(2n+1)(p-1)}$, and we get the 
inequalities
$$ \frac{1}{n(p-1)+p-2} \leq \delta(p,p^{n+1}-2) \leq \frac{2}{(2n+1)(p-1)}$$

When $p=3$, we get the density. Assume $p\geq 5$ for a while, and the right hand inequality is strict. Let $U$ denote a minimal irreducible 
solution, with length $\ell$ and weight $w$; it has density 
$\delta(p,p^{n+1}-2)=\frac{w}{\ell(p-1)}$, and we get the inequalities $nw+\frac{1}{2}w< \ell\leq nw+\frac{p-2}{p-1}w$. As a consequence, we must have 
$w\geq 3$, and we can write $\ell=nw+i$ for some $\frac{1}{2}w< i\leq \frac{p-2}{p-1}w$. From Proposition \ref{bou}, we deduce that the support of 
$U$ satisfies
$$|\varphi_U|\geq \frac{w(w+1)}{2}\frac{p^n-1}{p-1}+\frac{i(i+1)}{2}p^n >  \frac{w(w+1)}{2}\frac{p^n-1}{p-1}+\frac{w(w+2)}{8}p^n$$
Applying \cite[Lemma 1.2 (i)]{dens}, we get $(p-1)|\varphi_U|=\sum_D ds_p(u_d)\leq (p^{n+1}-2)w$, and putting this together gives 
$(p^{n+1}+3p^n-4)w\ < 6p^{n+1}-2p^n-12$, that is $w < 6$.

It remains to treat the cases $w\in \{3,4,5\}$ separately

\begin{itemize}
	\item $w=3$; in this case we must have $i=2$, and the inequality becomes $6(p^n-1)+3(p^{n+1}-p^n)\leq 3p^{n+1}-6$, which is impossible
	\item $w=4$, $i=3$ here we get $10(p^n-1)+6(p^{n+1}-p^n)\leq 4p^{n+1}-8$, once again impossible
	\item $w=5$, $i=3$ here we get $15(p^n-1)+6(p^{n+1}-p^n)\leq 5p^{n+1}-10$, once again impossible
	\item $w=5$, $i=4$ here we get $15(p^n-1)+10(p^{n+1}-p^n)\leq 5p^{n+1}-10$, finally impossible.
\end{itemize}

As a consequence, we have proven the first assertion of the following

\begin{proposition}
\label{cas1}
The $p$-density of the set $D:=\{1\leq i\leq p^{n+1}-2,~(i,p)=1\}$ is $\delta(p,p^{n+1}-2)=\frac{2}{(2n+1)(p-1)}$.
\begin{itemize}
	\item[(i)]  When $n\geq 1$, the minimal irreducible solutions all have length $2n+1$; 
up to shift, they are the 
$$p^n\cdot(p^{n+1}-i)+1\cdot(ip^n-1)=p^{2n+1}-1,~1\leq i\leq p-1$$
\item[(ii)] When $n=0$, the minimal irreducible solutions can have length $1$ or $2$. Up to shift, they are the 
$$ \left\{ \begin{array}{rcl}
i+(p-1-i)=p-1,~1\leq i\leq \frac{p-1}{2} & \textrm{for} & \ell=1;\\
~p\cdot(p-2)+d_1+d_2+d_3=p^2-1 & \textrm{for} & \ell=2\\
\end{array}\right.$$
for some $d_1,d_2,d_3\in D$, $d_1+d_2+d_3=2p-1$.
\end{itemize}
\end{proposition}

\begin{proof}
We come back to the general case $p\geq 3$. We first treat the case $n\geq 1$, and look for the minimal irreducible solutions; let $U$ denote 
one, having length $\ell$ and weight 
$w$. We must have $\ell=nw+\frac{w}{2}$, and $w$ must be even. Exactly as above, we get $w<6$, and the only remaining possibilities are $w=2$ or 
$w=4$.

We first treat the case $w=2$; here we have $\ell=2n+1$; from \cite[Lemma 1.4 (i)]{dens}, we get the inequality $(p-1)|\varphi_U|\leq 2p^{n+1}-4$. 
If the support $\varphi_U$ contains a geometric subsequence of length $\geq n+2$, then we get $(p-1)|\varphi_U|\geq p^{n+2}-1$, a contradiction 
with the preceding inequality. Thus the support of $U$ consists of two geometric subsequences of respective lengths $n+1$ and $n$. Call $n_1$ 
and $n_2$ their initial terms; then we have $n_1(p^{n+1}-1)+n_2(p^{n}-1)\leq 2p^{n+1}-4$, implying $n_1=1$ and $n_2(p^n-1)\leq p^{n+1}-3$, that 
is $1< n_2< p$, the strict inequalities coming from the irreducibility of $U$.

From the third assertion of Proposition \ref{gapgeom}, we deduce that the solution corresponding to the support 
$\{1,\ldots,p^n,n_2,\ldots,p^{n-1}n_2\}$ is $p^n(p^{n+1}-n_2)+1\cdot(p^nn_2-1)=p^{2n+1}-1$; this gives all the announced solutions.

It remains to show that there is no irreducible solution in the case $w=4$; we assume $\ell=4n+2$ and $w=4$. We consider four distinct cases,
according to the possible number of jumps $1\leq t\leq 4$ in its support.

 Assume first that there exists such a solution $U$, whose support is a support map with $4$ jumps. The jumps are contained in $D$, and
	bounded by $p^{n+1}-2$; from Lemma \ref{max1}, we must have 
	$$\varphi_U(i)\leq \frac{p^{n+1}-2}{p-1}=p^n+\cdots+p+1-\frac{1}{p-1}$$
	Let us denote by $n_k,\cdots,n_kp^{\ell_k-1}$, $1\leq k\leq 4$, the geometric subsequences of $\varphi_U$; we must have $\ell_k\leq n+1$ 
	from the inequality above, and moreover $\ell_k=n+1$ implies $n_k=1$. From the irreducibility of $U$, the $n_k$ must be pairwise distinct, 
	and we get that all $\ell_k$ must be less than or equal to $n$, except at most one, which is $n+1$. This contradicts the equality 
	$\sum \ell_k=4n+2$.

Now we consider a solution $U$ whose support has $t$ jumps, $1\leq t\leq 3$. Assume $t=3$; the jumps are the sums of at most two elements
	in $D$, and the maximal jump is at most $2(p^{n+1}-2)$. Applying Lemma \ref{max1}, 
	we get $\varphi_U(i)\leq 2(p^n+\cdots+p+1-\frac{1}{p-1})$ for any $i$. As above, we can have at most two geometric subsequences of length 
	$n+1$, with initial terms $1$ and $2$. The last geometric subsequence has length at most $n$, and we must have 
	$4n+2=\ell_1+\ell_2+\ell_3\leq 3n+2$, a contradiction. 
	
Now assume $t=2$; Lemma \ref{max1} gives the inequality 
	$\varphi_U(i)\leq 3(p^n+\cdots+p+1-\frac{1}{p-1})$. If $p=3$, one of the geometric subsequences can have length $n+2$, and initial term $1$, 
	while the other one has length at most $n+1$; we get $4n+2=\ell_1+\ell_2\leq 2n+3$, a contradiction. If $p\geq 5$, both subsequences have 
	length at most $n+1$, which gives another contradiction. 
	
	Finally, in the case $t=1$, the support of $U$ is a geometric sequence of length $4n+2$, with jump a sum of four elements in $D$.
	The inequality 
	$\varphi_U(i)\leq 4(p^n+\cdots+p+1-\frac{1}{p-1})$ from Lemma \ref{max1} shows that it has length at most $n+1$ ($n+2$ when $p=3$). 
	
	In any case we get a contradiction.
	
	It remains to treat the case $n=0$, i.e. $D=\{1,\cdots,p-2\}$. First assume we have $w=2$ and $\ell=1$; we get the inequality 
	$(p-1)|\varphi_U|\leq 2p-4$, and we must have $|\varphi_U|=1$. In this way we obtain the solutions of the first type. Now if we have 
	$w=4$ and $\ell=2$, we get the inequality $(p-1)|\varphi_U|\leq 4p-8$, and we must have $|\varphi_U|=3$, $\varphi_U=\{1,2\}$. A solution 
	with this support must have the form $\sum_{i=1}^4 p^{\varepsilon_i}d_i=p^2-1$ for some $\varepsilon_i\in \{0,1\}$ and $d_i\in D$, 
	with $\sum_{\varepsilon_i=0} p^{\varepsilon_i}d_i=p-2$ and $\sum_{\varepsilon_i=1} p^{\varepsilon_i}d_i=2p-1$. Since the sum of two elements
	in $D$ is at most $2p-4$, we can assume $\varepsilon_4=0$ and $\varepsilon_1=\varepsilon_2=\varepsilon_3=1$. We get all solutions of the 
	second form in this way.
\end{proof}

\begin{remark}
Compare assertion (ii) above with \cite[Theorem 3.8]{cfm}; this last result gives the minimal solutions of length one, and actually these are the only ones to be considered when $d<p-2$. But some length $2$ minimal solutions appear when $d=p-2$.
\end{remark}

We have proven the second and the third inequalities in Theorem \ref{bounds}. Actually when $n\geq 1$, assertion (i) of the proposition above shows that there exist solutions of density $\frac{2}{(2n+1)(p-1)}$ for the set $D$ as long as $d\geq p^{n+1}-p-1$. Assertion (ii) shows that there exist solutions of density $\frac{2}{p-1}$ for the set $D$ if, and only if $D$ contains both $i$ and $p-1-i$ for some $i\leq \frac{p-1}{2}$ if, and only if $d\geq \frac{p-1}{2}$.

\subsection{The case \texorpdfstring{$d=p^{n+1}-p-1$, $n\geq 2$}{ptrois}}
\label{pnplusone}

From the description of the minimal irreducible solutions for the set $\{1\leq i\leq p^{n+1}-2,~(p,i)=1\}$ and the prime $p$, we see that 
when $d=p^{n+1}-p-1$, there no longer exist solutions with density $\frac{2}{(2n+1)(p-1)}$. Thus we get the following bounds on the density of 
the new set $D=\{1\leq i\leq p^{n+1}-p-1,~(p,i)=1\}$
$$\frac{2}{(2n+1)(p-1)}<\delta(p,p^{n+1}-p-1)\leq \frac{1}{n(p-1)}.$$

Our aim here is to show that the right-hand inequality is actually an equality, and to describe the minimal irreducible solutions.

Let $U$ denote an irreducible solution with weight $w$ and length $\ell$; assume its density lies in the interval $\left]\frac{2}{(2n+1)(p-1)}, 
\frac{1}{n(p-1)}\right]$. Then we can write $\ell=nw+r$ where $0\leq r<\frac{w}{2}$.

From Proposition \ref{bou}, we deduce that the support of $U$ satisfies
$$|\varphi_U|\geq \frac{w(w+1)}{2}\frac{p^n-1}{p-1}+\frac{i(i+1)}{2}p^n >  \frac{w(w+1)}{2}\frac{p^n-1}{p-1}$$
Applying \cite[Lemma 1.2 (i)]{dens}, we get $(p-1)|\varphi_U|=\sum_D ds_p(u_d)\leq (p^{n+1}-p-1)w$, and putting this together gives 
$\frac{w+1}{2}(p^n-1)<p^{n+1}-p-1$, and $w<2p-1$. 

As in the preceding subsection, we consider different cases, according to the number of jumps in the support. Let $\ell_1,\ldots,\ell_w$ and $n_1,\ldots,n_w$ denote the respective lengths and initial terms of its geometric subsequences, with $\sum \ell_i=\ell=nw+r$.

\begin{lemma}
\label{case1}
Let $U$ denote a solution with length $\ell$ and weight $w$ such that $\ell=nw+r$, $0\leq r<\frac{w}{2}$. If the support of $U$
has $w$ jumps, then we have $r=0$.
\end{lemma}

\begin{proof}
Assume we have $r\geq 1$, and the support of $U$ has $w$ jumps. Note that we must have $w\geq 3$, and the jumps all lie in $D$ from Proposition 
\ref{gapgeom}. Shifting if necessary, we can assume $\ell_1=\max\{\ell_i\}$. From Lemma \ref{max1}, we have 
$p^{\ell_1-1}i_1\leq \frac{p^{n+1}-p-1}{p-1}$, and we get $\ell_1\leq n+1$, and $i_1=1$ 
if this is an equality. As a consequence, we have at most one geometric subsequence of length $n+1$, all other having length at most $n$. 
From the equality $\sum \ell_i=\ell=nw+r$, we must have $r=1$, $\ell_1=n+1$ and $\ell_2=\ldots=\ell_w=n$. Now we have 
$n_2\geq p^{n+1}-\max D=p+1$, and from Lemma \ref{max1}, $p^{n-1}n_2<\frac{p^{n+1}-p-1}{p-1}$. Thus we have $n_2=p+1$; since $w\geq 3$, 
we can consider $n_3$, and we have $p^nn_2-n_3=p^{n+1}+p^n-n_3\leq \max D$, and $n_3\geq p^n+p+1$, contradicting the inequality
$p^{n-1}n_3<\frac{p^{n+1}-p-1}{p-1}$ from Lemma \ref{max1}. As a consequence, there is no solution of density in 
$\left]\frac{2}{(2n+1)(p-1)}, \frac{1}{n(p-1)}\right[$, and whose support has $w$ jumps.

\end{proof}

In the following, we assume that the support of $U$ has $t$ jumps with $t<w$, and let $\ell_1,\ldots,\ell_t$ and $n_1,\ldots,n_t$ denote the respective lengths and initial terms of its geometric subsequences, with $\sum \ell_i=\ell=nw+r$. We also denote by $j_1,\ldots,j_t$ the jumps
of the support $\varphi_U$.

\begin{lemma}
\label{case2}
Let $U$ denote a solution with length $\ell$ and weight $w$ such that $\ell=nw+r$ and $0\leq r<\frac{w}{2}$; assume moreover that the support 
of $U$ has $t$ jumps, with $t<w$. Then the maximal length of a geometric subsequence in this support is $n+1$.
\end{lemma}

\begin{proof}

Assume the support of $U$ has $t$ jumps, $t<w$, 
and at least one geometric subsequence of length $\ell_1\geq n+2$. From \cite[Lemma 1.2 (i)]{dens} and the bound $w<2p-1$ we obtained above, 
we have
$$(p-1)\sum_{i=0}^{\ell-1} \varphi_U(i)=\sum_{k=1}^t (p^{\ell_k}-1)n_k\leq w(p^{n+1}-p-1) < 2p^{n+2}-2$$

As a consequence, we must have $\ell_1=n+2$, $\varphi_U(0)=n_1=1$ and $\ell_k\leq n+1$ for any $2\leq k\leq t$. Moreover, we get 
$\varphi_U(n+1)=p^{n+1}$, and the inequality
$$ (p-1)\sum_{i\neq n+1} \varphi_U(i)\leq (2p-2)(p^{n+1}-p-1)-(p-1)p^{n+1}<(p-1)p^{n+1}$$
Thus $\varphi_U(n+1)$ is the maximum of the $\varphi_U(i)$, and we obtain the following inequality $p\varphi_U(n+1)-\varphi(n+2)> (p-1)p^{n+1}$. 
The first jump must satisfy $j_{1}> \frac{(p-1)p^{n+1}}{p^{n+1}-p-1}>p-1$. We deduce that $s:=w-t=\sum(j_i-1)\geq p-1$. Moreover we have 
$\sum_{i>1} j_i\leq p-2$, and since we have $j_i\geq 1$ for any $i$, we find $j_i\leq p-t$ for any $i>1$.

As a consequence, we obtain, for any $n+2\leq i<\ell$, the inequality $\varphi_U(i+1)\geq p\varphi_U(i)-(p-t)(p^{n+1}-p-1)$; since we have 
$\varphi_U(\ell)=1<\varphi_U(i)$ for any $n+2\leq i<\ell$, we deduce (as in the proof of Lemma \ref{max1}) the inequality 
$\varphi_U(i)\leq (p-t)(p^{n+1}-p-1)/(p-1)$ for any $n+2\leq i<\ell$. Assume $\varphi_U(i)$ is the last term of a geometric subsequence
of length $n+1$; in this case we have $\varphi_U(i)=p^n\varphi_U(i-n)\leq (p-t)(p^{n+1}-p-1)/(p-1)$, and lastly the inequality $\varphi_U(i-n)\leq (p-t)(1+\frac{1}{p}+\cdots+\frac{1}{p^{n-1}}-\frac{1}{(p-1)p^n})<(p-t)(1+\frac{1}{p-1})$. If we have $t>1$, then we get $\varphi_U(i-n)< p-t+1$; for 
$t=1$, we get $\varphi_U(i-n) \leq p$, but equality is impossible since we assumed the solution $U$ irreducible, and we already have $\varphi_U(1)=p$. We get $n_k\leq p-t$, and there 
are at most $p-t-1$ subsequences of length $n+1$ (no one can begin with $n_k=1$ since $U$ is irreducible and $n_1=1$). This gives the inequality

$$\ell\leq n+2+(p-t-1)(n+1)+\left(t-1-(p-t-1)\right)n$$

the first term coming from the subsequence of length $n+2$, the second from the ones of length $n+1$, and the last from
the remaining ones, of length at most $n$.

On the other hand, we have $\ell=nw+r=n(t+s)+r\geq nt+n(p-1)+r$. Comparing both inequalities, we must have $(n-1)(p-1)\leq 2-r-t$. 
Since we have $n\geq 2$, and $t\geq 1$, this is clearly impossible.

\end{proof}

\begin{lemma}
\label{case3}
A solution satisfying the conclusion of the Lemma above must verify $r=0$, and its support has $w$ jumps.
\end{lemma}

\begin{proof}
We must have $\ell_i\leq n+1$ for all $i$. Set $s:=w-t$; all jumps are less than or equal to $(s+1)\max D$, and 
Lemma \ref{max1} gives the upper bound 
$p^{\ell_k-1}n_k\leq (s+1)\frac{p^{n+1}-p-1}{p-1}$ for any $1\leq k\leq t$. Since the $n_k$ are pairwise distinct, the number of geometric 
subsequences with $\ell_k=n+1$ is at most 
$$\left\lfloor(s+1)\frac{p^{n+1}-p-1}{p^n(p-1)}\right\rfloor= \left\lfloor(s+1)\left(\sum_{i=0}^{n-1} \frac{1}{p^{i}}-
\frac{1}{p^{n}(p-1)}\right)\right\rfloor<\frac{p(s+1)}{p-1}.$$

From the equality $\sum \ell_i=nw+r=nt+(ns+r)$, we must have at least $ns+r$ subsequences of length $n+1$. As a consequence, we get 
$\frac{p(s+1)}{p-1}> ns+r\geq ns+1$.

First assume $r>0$; since $n\geq 2$, we get $\frac{p(s+1)}{p-1}> 2s+1$, and $\frac{p}{p-1}> 2-\frac{1}{s+1}\geq \frac{3}{2}$, 
which is impossible for odd $p$.

When $r=0$, we get the inequalities $\frac{3}{2}\geq \frac{p}{p-1}>n\frac{s}{s+1}\geq \frac{n}{2}$, and we are reduced to the case $n=2$. Then we obtain the inequality $s<\frac{p}{p-2}$ and we must have $s=1$, or $s=2$ and $p=3$. 

The last case corresponds to the set $D=\{1\leq i\leq 23,~(i,3)=1\}$ when $p=3$; an exhaustive calculation gives the minimal irreducible solutions. 
These are the ones given in the next proposition, and all have their support with $w$ jumps. 

In the case $n=2$, $s=1$, the support has length $2w$, weight $w$, and consists of 
$w-1$ geometric subsequences of length at most $3$; moreover the maximal jump is at most $(s+1)\max D=2(p^3-p-1)$. From Lemma \ref{max1}, the elements in
the support of $U$ all satisfy $\varphi_U(i)\leq 2\frac{p^3-p-1}{p-1}=2\left(p^2+p-\frac{1}{p-1}\right)$, and there exist at most two geometric 
sequences of length $3$. If we denote by $t_i$ the number of geometric subsequences of length $i$ for $1\leq i\leq 3$, we have $\sum t_i=w-1$, 
and $\sum it_i=2w$. From these, we get $t_3=2+t_1\geq 2$, $w\geq 3$, and there exist at least two geometric subsequences of length $3$. As a 
consequence, there exist two geometric sequences of length $3$, and the $w-3$ remaining ones have length $2$. Thus there exists some $k$ such that $\ell_k=3$, and $n_k\geq 2$; since the maximal jump is at most $(s+1)\max D=2(p^3-p-1)$, we get $n_{k+1}\geq p^3n_k-2(p^3-p-1)\geq 2p+2$. 
Recall the inequality $\varphi_U(i)\leq 2\frac{p^3-p-1}{p-1}=2\left(p^2+p-\frac{1}{p-1}\right)$ for any element in 
the support of $U$; since all lengths are at least $2$, $pn_{k+1}$ lies in the support, with $pn_{k+1}\geq 2p^2+2p$, a contradiction. 

\end{proof}

Summarizing the results of the three lemmas above, we have proven

\begin{proposition}
\label{wgeom}
The set $D=\{1\leq i\leq p^{n+1}-p-1,~(p,i)=1\}$ has $p$-density $\delta(p,p^{n+1}-p-1)=\frac{1}{n(p-1)}$. Moreover, the support of any minimal irreducible solution 
with weight $w$ has $w$ jumps.

\end{proposition}

It remains to write down the minimal solutions. From above, we are looking for the irreducible solutions of weight $w$ and length $\ell=nw$. 

\begin{proposition}
\label{minirr1}
Let $D$ be as above. Then the minimal irreducible solutions have weight $w\leq p$. 

The minimal solutions of weight $1$ are (up to shift) the $1\cdot n_0(p^n-1)=n_0(p^n-1)$ for $1\leq n_0\leq p-1$. For any $2\leq w\leq p-1$, the 
solutions of weight $w$ are, up to shift, of one of the two following types
\begin{itemize}
	\item[(i)] $\sum_{k=0}^{w-1} p^{nk}\cdot(n_{k+1}p^n-n_{k})=n_0\cdot(p^{nw}-1)$, where $n_w=n_0=\min\{n_k\}$, and all $n_k$ are pairwise distinct 
	elements in $\{1,\ldots,p-1\}$;
	\item[(ii)] $\sum_{k=0}^{w-3} p^{nk}\cdot(n_{k+1}p^n-n_{k})+p^{n(w-2)}\cdot(p^{n-1}n_{w-1}-n_{w-2})+p^{n(w-1)-1}\cdot(p^{n+1}-n_{w-1})=
	1\cdot(p^{nw}-1)$, where $n_0=1$, all $n_k$, $1\leq k\leq w-2$ are pairwise distinct elements of $\{2,\ldots,p-1\}$, and we have 
	$p+1\leq n_{w-1}\leq p^2-1$, $(p,n_{w-1})=1$.
\end{itemize}
When $w=p$, all solutions are of the second type from above.
\end{proposition}

\begin{proof}
Fix some integer $w\geq 1$; we are looking for all irreducible solutions of length $nw$ and weight $w$. From Proposition \ref{wgeom}, 
we are looking for solutions having a support with $w$ jumps. From Proposition \ref{gapgeom}, they are completely determined by their 
support $\varphi_U$, and we can focus on these last ones. Write such a support
$$n_1,\ldots,p^{\ell_1-1}n_1,\ldots,n_w,\ldots,p^{\ell_w-1}n_w$$
From Lemma \ref{max1}, all elements above must satisfy $\varphi_U(i)\leq \frac{p^{n+1}-p-1}{p-1}<p^n+p^{n-1}+\cdots+p$. As a consequence, any geometric subsequence of the support must have length $\ell_j\leq n+1$, and if one has length $\ell_j=n+1$, it is unique and has initial term $n_j=1$. We consider separately the minimal solutions, according to the existence, or not, of such a subsequence.

First assume (up to shift) that we have $\ell_1=n+1$, and thus $n_1=1$, and $\ell_i= n$ for any $2\leq i\leq w$, except one which is $n-1$. From Proposition \ref{gapgeom}, we must have $p^{n+1}-n_2\in D$, and we get $n_2\geq p+1$. 
	
If we assume $\ell_2=n$, we get $p^{\ell_2}n_2\geq p^{n+1}+p^n$, and $n_3\geq p^{\ell_2}n_2-(p^{n+1}-p-1)=p^n+p+1$. From Lemma \ref{max1}, we must have $\ell_3-1=0$, and again from Proposition \ref{gapgeom} we have $n_4\geq p(p^n+p+1)-(p^{n+1}-p-1)=(p+1)^2$. Finally we get $\ell_4=n$, and $p^{\ell_4-1}n_4=p^{n-1}n_4\leq p^n+p^{n-1}+\cdots+p$, a contradiction.

As a consequence, we must have $\ell_2=n-1$, and $\ell_3=\ldots=\ell_w=n$. Moreover, the integers $p^nn_w-1,p^nn_{w-1}-n_w,\ldots,p^nn_3-n_4$ all lie in $D$ from Proposition \ref{gapgeom}, and step by step we get $n_w,n_{w-1},\ldots,n_3 \in \{2,\ldots,p-1\}$, pairwise distinct since $U$ is assumed irreducible; in the same way, the integer $p^{n-1}n_2-n_3$ lies in $D$, and we get the inequality $n_2<p^2$ . This gives us all solutions of the second type above.

Now assume $\ell_i= n$ for any $1\leq i\leq w$; we can write the support
  $$n_1,\ldots,p^{n-1}n_1,\ldots,n_w,\ldots,p^{n-1}n_w$$
Write $n_1=\min\{n_k\}$ after shifting if necessary. From Lemma \ref{max1}, we must have $p^{n-1}n_k<p^n+p^{n-1}+\cdots+p$ for all $p$, and $n_k\leq p+1$. The case $n_k=p+1$ is impossible since then we shoud have $n_{k+1}\geq p^{n+1}+p^n-(p^{n+1}-p-1)=p^n+p+1$, and $p^{n-1}n_{k+1}>p^n+p^{n-1}+\cdots+p$, a contradiction with Lemma \ref{max1}. The case $n_k=p$ is also impossible, since then we would have $p^nn_{k-1}-p\in D$ from Proposition \ref{gapgeom}, and this number is divisible by $p$.
  
  Thus all $n_k$ are in $\{1,\ldots,p-1\}$, they must be pairwise distinct since $U$ is irreducible, and we get all solutions of the first type above.

\end{proof}

We have shown the first inequality of Theorem \ref{bounds} in the case $n\geq 2$ since any solution of density $\frac{1}{n(p-1)}$ has the form $\sum du_d$ where $u_d\geq 1$ for some $d\geq p^n-1$.

\subsection{The case \texorpdfstring{$d=p^{2}-p-1$}{pquatre}}

The remaining case is $D=\{1\leq i\leq p^2-p-1,~(i,p)=1\}$. Many of the results above remain valid, but not all; for instance the 
first assertion in Proposition \ref{wgeom} remains true, but not the second one. As a consequence, the solutions described in the beginning of
Proposition \ref{minirr1} remain minimal, but there are new ones.

The inequalities at the beginning of subsection \ref{pnplusone} remain valid: we get in the same way, for a minimal solution of length $\ell$ 
and weight $w$

$$\frac{2}{3(p-1)}<\delta(p,p^2-p-1)\leq \frac{1}{p-1},~w\leq 2p-2.$$

As a consequence, we have $\ell=w+r$, with $0 \leq r <\frac{w}{2}$.

We first consider Lemma \ref{case1}, and assume that the support of $U$ has $w$ jumps. Denote by $\ell_1,\ldots,\ell_w$ the lengths of the 
geometric subsequences with $\ell_1=\max\{\ell_i\}$. The maximal jump is at most $\max D=p^2-p-1$; from Lemma \ref{max1}, we obtain the 
inequality $p^{\ell_1-1}n_1\leq (p^2-p-1)/(p-1)<p$. As a consequence, we have $\ell_i=1$ for all $i$, and $r=0$. Thus an irreducible solution 
as above whose support has $w$ jumps verifies $\ell=w$.

Now consider Lemma \ref{case2}, and assume that the support of $U$ contains at least one subsequence of length at least $3$. Exactly as in the 
beginning of the proof of Lemma \ref{case2} (we do not use the hypothesis $n\geq 2$ there), we get at most one subsequence of length $3$, with 
first term $1=\varphi_U(0)$. In the same way, we have $\varphi_U(2)=p^2=\max\{\varphi_U(i)\}$, $w_1>p-1$, $s=w-t\geq p-1$, 
$t\leq w+1-p\leq p-1$ and $w_i\leq p-t$ for any $i>1$. From the last inequality, we deduce that for all $3\leq i<\ell$, 
$\varphi_U(i)\leq (p-t)(p-\frac{1}{p-1})$. As a consequence, if we have $\varphi_U(i)=p\varphi_U(i-1)$ (i.e. when $\varphi_U(i)$ is the 
second term of a geometric subsequence of length $2$), we must have $\varphi_U(i-1)<p-t$, and since $\varphi_U(i-1)>1$, we get at most 
$p-t-2$ such subsequences. We deduce the following upper bound for the length $\ell$

$$\ell\leq 3+2(p-t-2)+t-1-(p-t-2)$$

From the inequality $\ell=w+r=t+s=r\geq p-1+t+r$, we deduce $t+r\leq 1$. Since $t\geq 1$ from the construction, we must have $t=1$ and $r=0$; 
the support of $U$ must have length $3$, and be the geometric sequence $\varphi_U(i)=p^i$ for $0\leq i\leq 2$. Moreover we get $w=3$, and 
$p^3-1=p\varphi_U(2)-\varphi_U(0)\leq 3(p^2-p-1)$, a contradiction. Thus a minimal irreducible solution cannot have a support containing a 
geometric subsequence of length $3$.

We end with Lemma \ref{case3}, which is no longer true. Consider a solution $U$ with length $\ell=w+r$, weight $w$, whose support contains 
$t<w$ jumps. As in the proof, we set $s=w-t$; all jumps are at most $(s+1)\max D$, and Lemma \ref{max1} gives the upper bound 
$\varphi_U(i)\leq (s+1)(p-\frac{1}{p-1})<p(s+1)$. As a consequence, there are at most $s$ geometric subsequences of length $2$ 
in the support of $U$, and we get the inequality $\ell\leq 2s+t-s=t+s$. Now since we have $\ell=w+r=t+s+r$, we get $r=0$.

Note that in any case we conclude $r=0$; thus a minimal solution must have $\ell=w$, and we get

\begin{proposition}
\label{densp2}
The set $D=\{1\leq i\leq p^2-p-1,~(i,p)=1\}$ has $p$-density $\delta(p,p^2-p-1)=\frac{1}{p-1}$.
\end{proposition}

We now describe the minimal irreducible solutions. Let $U$ denote such a solution, with length $\ell$ and weight $w=\ell$. From the calculations
preceding Proposition \ref{densp2}, if we denote by $\ell_1,\ldots,\ell_t$ the lengths of the geometric subsequences in the support of $U$, we 
have $\ell_i\leq 2$. If we set $s:=w-t$, we must have $s$ geometric subsequences of length $2$, and $t-s$ ones of length $1$. Moreover, the 
discussion in the paragraph preceding Proposition \ref{densp2} ensures the $s$ geometric subsequences are $\{i,pi\}$ for $1\leq i\leq s$.

Since we have $\max |\varphi_U|\geq ps$, the maximal
jump satisfies the inequality $\max\{j_i\}\geq p(p-1)s$. Since this jump is the sum of $w_r$ elements in $D$ for some $r$, we get the inequality
$w_r\geq \frac{p(p-1)}{\max D}s>s$. On the other hand, we know that exactly $t$ among the $w_r$ are non zero; from the equality 
$\sum_{r=1}^\ell w_r=w=t+s$ we deduce $\max\{w_r\}=s+1$, and $w_r\in\{0,1\}$ for all other $r$.

Assume $s\geq 2$. Up to shift, we can suppose that $\varphi_U(0)=1$, and $\varphi_U(k_0)=2$ for some $0<k_0<\ell$. We consider two cases, according
to the place where the maximal jump $w_{i_0}$ occurs. If it occurs for some $i\in \{0,\ldots,k_0-1\}$, then for any $k_0\leq i\leq \ell-1$, we have
$p\varphi_U(i)-\varphi_U(i+1)\leq p^2-p-1$, and $\varphi_U(i+1)\geq p\varphi_U(i)-p^2+p+1$. Since we have $\varphi_U(k_0+1)=2p$, we get an
increasing sequence, contradicting the assumption $\varphi_U(\ell)=\varphi_U(0)=1$. Else the maximal jump occurs for some 
$i\in \{k_0,\ldots,\ell-1\}$, and we get $\varphi_U(i+1)\geq p\varphi_U(i)-p^2+p+1$ for all $0\leq i<k_0$. From the equality $\varphi_U(1)=p$,
we get once again an increasing sequence, and a contradiction.

Thus we have $s\in \{0,1\}$; we treat separately the two cases.

First assume we have $s=0$. Then the support has $w$ jumps, and it determines completely the solution $U$ from Proposition \ref{gapgeom}. Write the 
support $n_1,\ldots,n_w$; Lemma \ref{max1} ensures we have $1\leq n_k\leq p-1$. Finally all solutions are of the first type given in Proposition
\ref{minirr1}.

If we have $s=1$, the support contains $w-1$ jumps, and, up to shift, we can write it $1,p,n_2,\ldots,n_{w-1}$. First remark that we have
$w_{r_0}=2$ for some $r_0$, and all other $w_r$ equal $1$. As a consequence, the maximal jump is $2(p^2-p-1)$, and Lemma \ref{max1} gives
$\max \varphi_U\leq 2p-1$.

If we have $w_{\ell-2}=1$ then we get $\varphi_U(2)\geq p^2-(p^2-p-1)=p+1$. If moreover $w_{\ell-3}=1$, then $\varphi_U(3)\geq p^2+p-(p^2-p-1)=2p+1$, 
contradicting the inequality above. Thus we must have $2\in \{w_{\ell-3},w_{\ell-2}\}$, and $w_r=1$ for any $0\leq r\leq \ell-4$. We treat the 
two cases separately.

First assume $w_{\ell-2}=2$. If we have $\ell=w=2$, then the solution must be $d_1+d_2=p^2-1$ for some $d_1,d_2\in D$. Assume $\ell\geq 3$; since we have
$w_r=1$ for any $0\leq r\leq \ell-3$, we deduce $pn_k-n_{k+1}\in D$ for all $2\leq k\leq \ell-1$. In particular, we have $pn_{\ell-1}-1\in D$, and
$n_{\ell-1}\in \{2,\ldots,p-1\}$. Reasoning (recursively) the same way for all $pn_k-n_{k+1}\in D$, we get $n_{k}\in \{2,\ldots,p-1\}$ for all 
$2\leq k\leq \ell-1$. Summing up, we deduce the solutions
$$p^{\ell-2}\left(d+(p^2-d-n_2)\right)+\sum_{k=2}^{\ell-1}p^{\ell-1-k}(pn_k-n_{k+1})=p^w-1$$
where $d$ and $p^2-d-n_2$ are both in $D$, $n_2,\ldots,n_{\ell-1}$ are pairwise distinct elements in $\{2,\ldots,p-1\}$, and $n_\ell=1$.

We now assume $w_{\ell-3}=2$. We must have both $\ell=w\geq 3$, and $n_2\geq p^2-(p^2-p-1)=p+1$.

When $\ell=w=3$, the support must be $1,p,n_2$, with $pn_2-1\leq 2(p^2-p-1)$, i.e. $n_2<2p-2$. In this way, we get all solutions of the form 
$$p(p^2-n_2)+d_1+d_2,~d_1,d_2\in D,~d_1+d_2=pn_2-1$$
When $\ell=w> 3$, we have $w_r=1$ for all $0<r<\ell-3$, and we deduce as above that $n_3,\ldots,n_{\ell-1}$ are pairwise distinct elements in
$\{2,\ldots,p-1\}$. From this support, we deduce the solutions 
$$p^{\ell-2}(p^2-n_2)+p^{\ell-3}(d_1+d_2)+\sum_{k=3}^{\ell-1}p^{\ell-1-k}(pn_k-n_{k+1})=p^w-1$$
where we have $p+1\leq n_2\leq 2p-2$, $d_1,d_2\in D$ and $d_1+d_2=pn_2-n_3$.

We have determined all minimal irreducible solutions

\begin{proposition}
\label{minirr2}
Let $p$ denote an odd prime, and $D:=\{1\leq i \leq p^2-p-1,~ (i,p)=1\}$. The minimal irreducible solutions have one of the following three forms
\begin{itemize}
\item[(i)] $\sum_{k=0}^{w-1} p^{k}\cdot(n_{k+1}p-n_{k})=n_0\cdot(p^{w}-1)$, where $n_w=n_0=\min\{n_k\}$, and all $n_k$ are pairwise distinct elements in $\{1,\ldots,p-1\}$;
\item[(ii)] $p^{\ell-2}\left(d+(p^2-d-n_2)\right)+\sum_{k=2}^{\ell-1}p^{\ell-1-k}(pn_k-n_{k+1})=p^w-1$
where $d$ and $p^2-d-n_2$ are both in $D$, $n_2,\ldots,n_{\ell-1}$ are pairwise distinct elements in $\{2,\ldots,p-1\}$, and $n_\ell=1$.
\item[(iii)] $p^{\ell-2}(p^2-n_2)+p^{\ell-3}(d_1+d_2)+\sum_{k=3}^{\ell-1}p^{\ell-1-k}(pn_k-n_{k+1})=p^w-1$ where we have else
\begin{itemize}
\item[($\alpha$)] $\ell=3$ and $p+1\leq n_2< 2p-2$;
\item[($\beta$)] $\ell\geq 4$, $p+1\leq n_2\leq 2p-2$, $d_1,d_2\in D$ with $d_1+d_2=pn_2-n_3$, $n_3,\ldots,n_{\ell-1}$ are pairwise elements in $\{2,\ldots,p-1\}$, and $n_\ell=1$.
\end{itemize}  
\end{itemize}
\end{proposition}

As a consequence of this proposition, we have proven the remaining case of Theorem \ref{bounds}, i.e. the first inequality in the case $n=1$.

\section{Some results about Artin-Schreier curves}
\label{sec4}

In this section, we study Artin-Schreier curves, i.e. $p$-cyclic covering of the projective line in characteristic $p$. We shall concentrate on $p$-rank $0$ such curves. From the Deuring Shafarevic formula \cite[Corollary 1.8]{crew}, they are ramified at exactly one point, and moving it to infinity, we deduce that such a curve, defined over $k=\F_q$, has an equation of the form
$$y^p-y=f(x)=\sum_{i=1}^d c_ix^i,~ c_i\in k$$
In the following, we denote this curve by $C_f$. Let us first describe its zeta function. It is well known that it has the form
$$Z(C_f,T)=\frac{L(C_f,T)}{(1-T)(1-qT)}$$
where $L(C_f,T)$ is a polynomial in $\Z[T]$ of degree two times the genus of the curve $C_f$, $2g(C_f)=(p-1)(d-1)$. We denote by $\NP_q(C_f)$ its $q$-adic Newton polygon, this is the Newton polygon of the curve $C_f$. It is a convex polygon with end points $(0,0)$ and $(2g,g)$, slopes in $[0,1]$, and symmetric in the sense that for any segment of length $\ell$ and slope $s$, there is a segment of length $\ell$ and slope $1-s$.

We now recall the link between exponential sums and Artin-Schreier curves. For any integer $r\geq 1$, we denote by $k_r=\F_{q^r}$ the degree $r$ extension of $k$ inside a fixed algebraic closure of $k$. We extend the additive character $\psi$ to $k_r$ with the help of the trace; in this way we obtain an additive character $\psi_{mr}:=\psi\circ \Tr_{k_r/k}$ of the field $k_r$. For any one variable polynomial $f\in k[x]$, we define a family $(S_r(f))_{r\geq 1}$ of exponential sums and the associated $L$-function in the following way
$$S_r(f):=\sum_{x\in k_r} \psi_{mr}(f(x)),~ L(f,T):=\exp\left(\sum_{r\geq 1} S_r(f)\frac{T^r}{r} \right)$$
The $L$-function is a polynomial of degree $d-1$ in $\Z[\zeta_p][T]$, and the polynomial $L(C_f,T)$ factors as
$$L(C_f,T)=\textrm{N}_{\ma{Q}(\zeta_p)/\ma{Q}}\left( L(f,T)\right)$$
 Since the prime $p$ is totally ramified in the extension $\Q(\zeta_p)/\Q$, we get the expression $\NP_q(C_f)=(p-1)\NP_q(f)$ for the Newton polygons, where $\NP_q(f)$ denotes the Newton polygon of the $L$-function $L(f,T)$. 

From Grothendieck's specialization theorem, there exists a polygon $\GNP(d,p)$, the generic Newton polygon, such that when $f$ varies among degree $d$ polynomials with coefficients in $\overline{k}$, the polygon $\NP_q(C_f)$ (here $\F_q$ is the field of definition of $f$) lies above $\GNP(d,p)$, and they are generically equal. From \cite{dens}, we can reinterpret the bounds in Theorem \ref{bounds} as the first slopes of the generic Newton polygons $\GNP(d,p)$.  

In order to prove Theorem \ref{supsing}, we have to be more precise; we show that when the degree has the form $i(p^n-1)$, $n\geq 1$, $1\leq i\leq p-1$, the first slope of the Newton polygon $\NP(f)$ is $\frac{1}{n(p-1)}$ \emph{for any} polynomial $f$. 

Finally, we compute the first vertex of the polygon $\GNP(d,p)$, and the corresponding Hasse polynomial; this is the polynomial in the coefficients of $f$ whose non vanishing ensures that $\NP_q(C_f)$ and $\GNP(d,p)$ share the same first vertex.

Our main tool is the $p$-adic congruence for the $L$-function associated to $f$ in terms of characteristic polynomials of semi-linear endomorphisms given in \cite{cong}. If we set $\delta:=\delta(d,p)$, it is a congruence "along the first slope", in the ring $M_\delta$ of power series having first slope at least $\delta$ modulo the ideal $I_\delta$ of power series having first slope greater than $\delta$. Explicitely, let $\Gamma:=(\gamma_i)_{1\leq i\leq d}$, where $\gamma_i$ is the Teichm\"{u}ller lifting of $c_i$; we have
\begin{equation}
\label{cong}
L(f,T)\equiv \det\left(\textbf{I}-\pi^{m(p-1)\delta}TM(\Gamma)^{\tau^{m-1}}\cdots M(\Gamma)\right) \mod I_{\delta}
\end{equation}
where $M(\Gamma)$ is a matrix whose coefficients are in the ring $\Z_p[\Gamma]$, and the number $\pi$ is the solution of the equation $X^{p-1}+p=0$ in the $p$-adic ring $\Z_p[\zeta_p]$ defined by $\psi(1)\equiv 1+\pi \mod \pi^2$. Since the matrix $M(\Gamma)$ can be written explicitely once one knows the minimal irreducible solutions, we can conclude.

 We first give some elementary results about such endomorphisms of a finite dimensional vector space over a finite field; this allows us to give the degree of the polynomial in the right hand side of the above congruence in the generic case. Then we prove that in certain cases this degree cannot be zero, which guarantees the non-existence of supersingular Artin-Schreier curves for certain genera when $\delta<\frac{1}{2}$. Finally, since the length of the first slope depends on the degree of the characteristic polynomial above, and the Hasse polynomial is its leading coefficient, we are able to give the Hasse polynomial for the first vertex of the generic Newton polygon for any degree $d$.
 
 \medskip
 
{\bf Notation :} in the following, we denote by $\{f^i\}_j$ the degree $j$ coefficient of the polynomial $f^i$.

\subsection{Semi-linear endomorphisms}

We first give a technical result on the characteristic polynomials of iterates of a semi-linear endomorphism, that we use below; this result must be well-known, but we give a proof here since we did not find any appropriate reference. Note the idea is already present in \cite{ko}, where the author considers Hasse Witt matrices.

Set $k:=\F_q$, $q=p^m$, and denote by $\sigma$ the generator of Gal$(k/\F_p)$ that raises the elements to the $p$-power. We consider the $\sigma$-linear endomorphism $\varphi$ of $V=k^N$ having matrix $^tA$ with respect to some basis. Then the matrix of $\varphi^m$, which is $k$-linear, is $^tB$, where $B=A^{\sigma^{m-1}}\cdots A$; we get the equality 
$$\det\left(\I-TA^{\sigma^{m-1}}\cdots A\right)=\det\left(Id-T\varphi^m\right)$$

We follow \cite{ka2}, and define the subspaces $V_{ss}:=\cap_{n\geq 1} \Ima\varphi^n$, $V_{nil}:=\cup_{n\geq 1} \Ker\varphi^n$. We have the direct sum $V=V_{ss}\oplus V_{nil}$; let $\{v_1,\ldots,v_{s_0}\}$ denote a basis of $V_{ss}$. The subspaces $V_{ss}$ and $V_{nil}$ are stable by $\varphi$. We show

\begin{lemma}
\label{semilincharpol}
Assume that we have $V_{ss}\neq 0$; let $A_{ss}$ denote the matrix for the restriction of $\varphi$ to $V_{ss}$ relative to the basis $\{v_1,\ldots,v_{s_0}\}$. We have the equality of polynomials
$$\det\left(\I-TA^{\sigma^{m-1}}\cdots A\right)=\det\left(\I-TA_{ss}^{\sigma^{m-1}}\cdots A_{ss}\right).$$
Moreover this polynomial has degree $\dim_k V_{ss}$ and leading coefficient $\textrm{N}_{k/\ma{F}_p}(\det A_{ss})$.
\end{lemma}

\begin{proof}
Consider the following sequence of vector subspaces of $V$ 
$$\Ker\varphi\subsetneq\Ker\varphi^2\subsetneq\ldots\subsetneq\Ker\varphi^t=V_{nil}$$
For any $0\leq i\leq t-1$, let $\{v_{s_i+1},\ldots,v_{s_{i+1}}\}$ be a basis for a complementary subspace of $\Ker\varphi^i$ in $\Ker\varphi^{i+1}$. From the construction, the family $\{v_1,\ldots,v_{s_t}\}$ is a basis for the $k$-vector space $V$, and the matrices of $\varphi$ and $\varphi^m$ in this basis have the block forms
$$\left(\begin{array}{cc}
A_{ss}  & 0  \\
0  & T_1 \\
\end{array}
\right)
\textrm{ and }
\left(\begin{array}{cc}
A_{ss}\cdots A_{ss}^{\sigma^{m-1}}   & 0  \\
0  & T_2 \\
\end{array}
\right)$$

where the $T_i$ are strictly upper triangular matrices. We get the equalities of the reciprocals of the characteristic polynomials since a matrix and its transpose have the same one.

The last assertions follow from the fact that $A_{ss}$ is an invertible matrix; actually the definition of $V_{ss}$ ensures that the restriction of $\varphi$ to this subspace is surjective, thus an isomorphism.
\end{proof}

From this result and the congruence above, we deduce the first vertex of the Newton polygons interms of the invariants defined above

\begin{corollary}
\label{firstvertex}
Let $\overline{M}(\Gamma)$ denote the reduction modulo $p$ of the matrix $M(\Gamma)$, and $V_{ss}$ the space associated to $\overline{M}(\Gamma)$ as above. Assume $V_{ss}\neq \{0\}$; then the first vertex of the Newton polygon $\NP_q(f)$ is $(\dim V_{ss},\delta\dim V_{ss})$, and the first vertex of the Newton polygon $\NP_q(C_f)$ is $((p-1)\dim V_{ss},(p-1)\delta\dim V_{ss})$.
\end{corollary}

\begin{proof}
Write $\det(\textbf{I}-T\overline{M}(\Gamma)^{\sigma^{m-1}}\cdots \overline{M}(\Gamma)):=\sum_{i=0}^k a_iT^i$ in $\F_p[T]$, with $k=\dim V_{ss}$ from the above Lemma. The congruence means that if we write $L(f,T)=\sum_{i=0}^{d-1} b_iT^i$, we have $b_i=\pi^{m(p-1)\delta i}b_i'$ where the reduction modulo $p$ of $b_i'$ is $a_i$. As a consequence the coefficient with highest degree such that $v_q(b_i)=\delta i$ is the coefficient of degree $k$.

The assertion about $\NP(C_f)$ follows directly from the fact that this last Newton polygon is the dilation of $\NP_q(f)$ by a factor $p-1$.
\end{proof}

\subsection{Proof of Theorem \ref{supsing}}
\label{prove}

We are ready to prove the result about supersingular curves. Note that the case $p=2$ is \cite[Theorem 1.2]{sz}.

\begin{proof}
First note that a supersingular curve must have $p$-rank $0$. Now from the Deuring Shafarevic formula \cite[Corollary 1.8]{crew}, a $p$-cyclic covering of the projective line having $p$-rank $0$ must have exactly one ramification point. Moving it to infinity, we can assume that such a curve has a model of the form 
$$y^p-y=f(x)$$
where $f(x):=\sum_{j=0}^d c_jx^j$ is a degree $d$ polynomial from the Hurwitz' formula in positive characteristic.

In order to prove the Theorem, and from the description of the numerator of the zeta function of the curve in the Introduction, it is sufficient to show that the $q$-adic Newton polygon of the polynomial $L(f,T)$ has first slope $\frac{1}{n(p-1)}$ for any $f$. From corollary \ref{firstvertex}, we are reduced to show that the subspace $V_{ss}$ associated to the matrix $M(\Gamma)$ is never trivial.

We first treat the case $n\geq 2$.  We apply Proposition \ref{minirr1}. The minimal irreducible solutions for the set $D=\{1,\ldots,d\}$ cannot be of type (ii): since we have $p^{n+1}-(p^2-1)>(p-1)(p^n-1)$, there is no element in $D$ of the form $p^{n+1}-n_{w-1}$ with $p+1\leq n_{w-1}\leq p^2-1$. Thus all are of type (i). Moreover, the integer $p^n n_k-n_{k+1}$ is in $D$ if, and only if $n_k<i$, or $n_{k+1}\geq n_k=i$. As a consequence, the minimal irreducible solutions are the ones described in (i) of the above Proposition, with $1\leq n_0,\ldots,n_{w-1} \leq i-1$, and support $n_0,\ldots,p^{n-1}n_0,n_{w-1},\ldots,p^{n-1}n_{w-1},\ldots,n_{1},\ldots,p^{n-1}n_1$, and the solution $1\cdot i(p^n-1)$, with support $i,ip,\ldots,ip^{n-1}$. We deduce the minimal support \cite[Definition 2.10]{cong} associated to
$p$ and $D$; it is the set $\{kp^j,~ 1\leq k\leq i,~ 0\leq j\leq n-1\}$.

In order to use congruence (\ref{cong}), we now have to consider the base $p$ digits of these solutions. 

Concerning the last one, the solution is $U=(u_d)_D$, where $u_{i(p^n-1)}=1$, and all other are zero. We get the vectors $V={\bf 1}_{i(p^n-1)}\in V(p^{n-1}i,i)\subset\{0,\ldots,p-1\}^{|D|}$ with coordinates $1$ at the $i(p^n-1)$th place and $0$ elsewhere, and ${\bf 0} \in V(p^{j-1}i,p^ji)\subset\{0,\ldots,p-1\}^{|D|}$ for any $0\leq j\leq n-1$. As explained above, any other solution has the form described in Proposition \ref{minirr1} (i) with $n_0,\ldots,n_{w-1}$ pairwise distinct in $\{1,\ldots,i-1\}$. In other words, we have $U=(u_d)$ where $u_d=p^{nk}$ for $d=n_{k+1}p^n-n_{k}$, and $0$ else, and its base $p$ digits are the vectors $V={\bf 1}_{n_{k+1}p^n-n_k}\in V(p^{n-1}n_{k+1},n_k)\subset\{0,\ldots,p-1\}^{|D|}$ for any $0\leq k\leq w-1$ and the zero vector ${\bf 0} \in V(p^{j-1}n_{k},p^jn_k)\subset\{0,\ldots,p-1\}^{|D|}$ for any $1\leq j\leq n-1$, $0\leq k\leq w-1$.

Summing up, we have completely determined the sets defined in \cite[Definition 2.13]{cong}
$$V(kp^j,k'p^{j'})=\left\{
\begin{array}{rcl}
\{ {\bf 1}_{k'p^n-k}\}       & \textrm{when} & j=0,~j'=n-1,~\textrm{ and } 
\left\{ 
\begin{array}{c}
1\leq k,k'\leq i-1 \\
\textrm{ or } \\
k=k'=i \\
\end{array}\right.\\
\{{\bf 0}\} & \textrm{when} & (k,j)=(k',j'+1) \\
\emptyset & \textrm{else} & \\
\end{array}
\right.$$

From \cite[Definition 3.6]{cong}, we deduce the matrix $\overline{M}(\Gamma)$. We write it with respect to the basis $\{e_{kj},~ 1\leq k\leq i,~ 0\leq j\leq n-1\}$ in lexicographic order, where $e_{kj}$ corresponds to the integer $kp^j$ in the minimal support. It is the $in\times in$ matrix consisting of the $n\times n$ blocks $M_{ab}$, $1\leq a,b\leq i$, where we have $M_{aa}=A(c_{a(p^n-1)})$, $M_{ab}=B(c_{ap^n-b})$ for $a\neq b$ in $\{1,\ldots,i-1\}$, and $M_{ai}=M_{ib}=\textbf{O}_n$ for $1\leq a,b \leq i-1$; note that we have set

$$A(c):=
\left(\begin{array}{cc}
0 & \textbf{I}_{n-1} \\
c & 0 \\
\end{array}\right),~ 
B(c):=
\left(\begin{array}{cc}
0 & \textbf{O}_{n-1} \\
c & 0 \\
\end{array}\right)
$$

We deduce that the subspace $W$ of $V$ generated by the last $n$ vectors $e_{ij},~ 0\leq j\leq n-1$ is stable by $\varphi$, and moreover, since the (leading) coefficient $c_{i(p^n-1)}$ cannot be zero, that $\varphi_{|W}$ is bijective. Thus $W$ is a subspace of $V_{ss}$, we have $V_{ss}\neq 0$, and the congruence is always non trivial. This ends the proof in the case $n\geq 2$.

We consider the case $n=1$ (and $p\geq 5$ here in order to have $n(p-1)\geq 3$). Since $i(p-1)\leq (p-1)^2<p^2-(2p-2)$, we cannot have solutions of type (iii) from Proposition \ref{minirr2}. When $i\leq\frac{p-1}{2}$, we no longer have solutions of type (ii) since they use an integer $d$ such that $d$ and $p^2-d-n_2\geq p^2-d-p+1$ are both in $D$. In this case all solutions are of type (i): we get the matrix with coefficients $c_{pa-b}$ for $1\leq a,b \leq i-1$, $c_{(p-1)i}$ when $a=b=i$, and $0$ else; then we conclude exactly the same way as in the case $n\geq 2$. Assume now we have $i\geq\frac{p+1}{2}$. In addition to type (i) solutions, there are new solutions, of type (ii), of the form 
$$p^{\ell-2}\left(d+(p^2-d-n_2)\right)+\sum_{k=2}^{\ell-1}p^{\ell-1-k}(pn_k-n_{k+1})=p^\ell-1$$
where $d$ and $p^2-d-n_2$ are both in $D$, $n_2,\ldots,n_{\ell-1}$ are pairwise distinct elements in $\{2,\ldots,i-1\}$, and $n_\ell=1$. Such a solution has support $(1,p,n_2,\ldots,n_{\ell-1})$. We deduce that the minimal support is $\{1,2,\ldots,i,p\}$. In addition to the $V(k,k')=\{{\bf 1}_{pk-k'}\}$ as above, we get $V(1,p)=\{{\bf 0}\}$, $V(k,p)=\emptyset$ for $k\geq 2$, and $V(p,k)=\{{\bf 1}_{d}+{\bf 1}_{p^2-d-k},~ d \textrm{ and } p^2-d-k\in D\}$. The matrix $\overline{M}(\Gamma)$ is 
$$\overline{M}(\Gamma)=
\left(\begin{array}{ccc|cc}
 &  &       & 0 & 1       \\
 & (c_{pa-b}) & & \vdots  & 0 \\
 &  &  & 0 & \vdots  \\
 \hline
0 & \ldots & 0 & c_{(p-1)i} & 0 \\
\frac{1}{2}\left\{f^2\right\}_{p^2-1} & \ldots & \frac{1}{2}\left\{f^2\right\}_{p^2-i+1} & 0 & 0 \\
\end{array}\right)$$
Once again, the semi-linear map acts bijectively on the subspace generated by the last but one vector since the coefficient $c_{i(p-1)}$ cannot be zero, and we conclude as above.
\end{proof}

\subsection{Hasse polynomials for the first vertex}
\label{hasse}

One can also use the matrix $M(\Gamma)$ to give the first vertex of the polygon $\GNP(d,p)$, and the associated Hasse polynomial. We do this in this subsection; we write down explicitely the semi-linear morphism $\varphi$ having matrix $\overline{M}(\Gamma)$, then we determine its semi-simple part in order to apply Corollary \ref{firstvertex}.

\subsubsection{$ d\leq p-2$} The case $d< \frac{p-1}{2}$ is \cite[Theorem 1.1]{sz2}. Then the first vertex of the generic Newton polygon is $(p-1,\lceil\frac{p-1}{d}\rceil)$, with Hasse polynomial $\left\{f^{\lceil\frac{p-1}{d}\rceil}\right\}_{p-1}$. The result of Proposition \ref{cas1} (ii) shows that this result extends to the range $\frac{p-1}{2}\leq d < p-2$.

In the case $d=p-2$, we can assume $p\geq 7$, since the curve is rational for $p=3$, and supersingular for $p=5$. From Proposition \ref{cas1} (ii), the minimal support is $\{1,2\}$, and the matrix $\overline{M}(\Gamma)$ can be written in the following way
$$\left(\begin{array}{cc}
\frac{1}{2}\left\{f^{2}\right\}_{p-1} & c_{p-2} \\
\frac{1}{6}\left\{f^{3}\right\}_{2p-1} & 0 \\
\end{array}\right)$$
We deduce that the first vertex of the generic Newton polygon $\GNP(p-2,p)$ is $(2(p-1),4)$, with Hasse polynomial $c_{p-2}\left\{f^{3}\right\}_{2p-1}$. Moreover, when this polynomial vanishes, the first vertex of $\NP(C_f)$ is $(p-1,2)$ exactly when we have $\left\{f^{2}\right\}_{p-1}\neq 0$.

\subsubsection{$p^{n}-1\leq d\leq p^{n+1}-p^2-1$, $n\geq 2$} Let us treat the case $d=i(p^n-1)$, $n\geq 2$; we have already calculated the matrix $\overline{M}(\Gamma)$ in this case. It turns out that the matrix has (up to sign) the same determinant as the $i\times i$ matrix having coefficients $m_{ab}=c_{p^na-b}$ when $a,b \in \{1,\ldots,i-1\}$ or $a=b=i$,  and $m_{ab}=0$ else. This is easily seen when developing the determinant successively with respect to the lines $L_i$, $i$ running over the integers not multiple of $n$.

We deduce that when this determinant (say $H(\Gamma)$) is non zero, the matrix $\overline{M}(\Gamma)$ is invertible, and we have $V=V_{ss}$; the characteristic polynomial has degree $\dim V=ni$, and the first vertex of the Newton polygon of the $L$-function is $(ni,\frac{i}{p-1})$. 

As a consequence, the first vertex of the generic Newton polygon $\GNP(i(p^n-1),p)$ is $((p-1)ni,i)$, with Hasse polynomial 
$$H(\Gamma)=c_{i(p^n-1)} \det\left(c_{p^na-b}\right)_{1\leq a,b \leq i-1}$$

When we have $d=ip^n-1$, $n\geq 2$, we have to replace the matrix above by $\left(c_{p^na-b}\right)_{1\leq a,b \leq i}$. When $d=ip^n-t$, $1\leq t<i$, we get the same matrix, with zeroes in place of the coefficients $c_{p^ni-b}$, $1\leq b<t$.

When $ip^n-1\leq d<(i+1)(p^n-1)$, $1\leq i\leq p-2$, all minimal irreducible solutions already appear when $d=i(p^n-1)$, and we get the same result as in the case $d=ip^n-1$. Finally, for $(p-1)p^n-1\leq d\leq p^{n+1}-p^2-1$, we get the same result as in the case $d=(p-1)p^n-1$.

\subsubsection{$p-1\leq d\leq (p-1)^2$} As we have seen in the proof of Theorem \ref{supsing}, when $d\leq \frac{p-1}{2}p-1$, the only solutions in Proposition \ref{minirr2} are of type (i), and we get the same result as in the case $p^n-1\leq d\leq p^{n+1}-p^2-1$, $n\geq 2$ above, namely that the Hasse polynomial is the determinant of the $i\times i$ matrix consisting of the coefficients $c_{pa-b}$ for $pa-b\leq d$ when $i(p-1)\leq d<(i+1)(p-1)$.

When $d\geq \frac{p-1}{2}p+1$, solutions of type (ii) appear, but not of type (iii); as in the proof of the Theorem, we get a new matrix, generically invertible, of size $(i+1)\times (i+1)$ when $i(p-1)\leq d<(i+1)(p-1)$, namely
$$\overline{M}(\Gamma)=
\left(\begin{array}{ccc|cc}
 &  &       & c_{p-i} & 1       \\
 & (c_{pa-b}) & & \vdots  & 0 \\
 &  &  & c_{p(i-1)-i} & \vdots  \\
 \hline
c_{pi-1} & \ldots & c_{pi-i+1} & c_{(p-1)i} & 0 \\
\frac{1}{2}\left\{f^2\right\}_{p^2-1} & \ldots & \frac{1}{2}\left\{f^2\right\}_{p^2-i+1} & \frac{1}{2}\left\{f^2\right\}_{p^2-i} & 0 \\
\end{array}\right)$$
and whose determinant is the Hasse polynomial.

\subsubsection{$p^{n+1}-p^2+1\leq d\leq p^{n+1}-p-1$, $n\geq 2$} We first consider the case $d=p^{n+1}-p-1$, $n\geq 2$. From Proposition \ref{minirr1}, the minimal support consists of the integers $ip^k$, $1\leq i\leq p-1$, $0\leq k\leq n-1$, of $p^n$, and of the $ip^k$, $p+1\leq i\leq p^2-1$ prime to $p$, $0\leq k\leq n-2$. The semi-linear map $\varphi$ acts on the associated basis by $\varphi(e(ip^k))=e(ip^{k+1})$ for $1\leq i\leq p-1$, $0\leq k\leq n-2$ and $p+1\leq i\leq p^2-1$ prime to $p$, $0\leq k\leq n-3$. Moreover we have
$$f(e(ip^{n-1}))=\sum_{j=1}^{p-1} c_{p^ni-j}e(j)\textrm{ for } i\geq 2,~ f(e(p^{n-1}))=\sum_{j=1}^{p-1} c_{p^n-j}e(j)+e(p^n)$$
$$f(e(p^n))=\sum_{i=p+1}^{p^2-1} c_{p^{n+1}-i}e(i),~f(e(ip^{n-2}))=\sum_{j=1}^{p-1} c_{p^{n-1}i-j}e(j)\textrm{ for } p+1\leq i\leq p^2-1$$
where $i$ has to be prime to $p$ in the last line.

Here we consider the subspave $V_1$ generated by the $e(ip^k)$, $1\leq i\leq p-1$, $0\leq k\leq n-1$, and the $\varphi^k(e(p^n))$, $0\leq k\leq n-1$. For any vector of the basis $e(i)$, one of its iterates under the action of $\varphi$ falls on $V_1$; this is clear for most of them, and we compute
$$\varphi^n(e(p^n))=\sum_{j=1}^{p-1} \left(\sum_{i=p+1}^{p^2-1} c_{p^{n+1}-i}^{p^{n-1}}c_{p^{n-1}i-j}\right)e(j):=\sum_{j=1}^{p-1}\theta_je(j)$$
which proves the assertion for the last $n$ vectors spanning $V_1$. We deduce that $V_{ss}$ is a subspace of $V_1$. Writing the matrix of the restriction of $\varphi$ in this basis, we get the $np\times np$ matrix consisting of the $n\times n$ blocks $M_{ab}$, $1\leq a,b\leq p$, defined by (notations are the same as in the preceding subsection)

$$M_{aa}=A(c_{ap^n-a}),~ 
M_{ab}=B(c_{ap^n-b}),~M_{pb}=B(\theta_b)\textrm{ for } 1\leq a,b\leq p-1,~ $$
and finally $M_{1p}=B(1)$, $M_{ap}={\bf 0}_n$ for $2\leq a\leq p-1$, and $M_{pp}=A(0)$.

Developping the determinant of this matrix with respect to the lines $L_i$, $i$ running over the integer not divisible by $n$, then by the last column of the remaining matrix, we see that its determinant is the same as the one of the following $(p-1)\times(p-1)$ matrix 

$$\left(
\begin{array}{ccc}
c_{2p^n-1} & \ldots & c_{2p^n-p+1}\\
\vdots & & \vdots\\
c_{(p-1)p^n-1} & \ldots & c_{(p-1)(p^n-1)} \\
\theta_1 & \ldots & \theta_{p-1}\\
\end{array}
\right)$$
This last determinant is generically non trivial (it is a non zero polynomial in the coefficients of the polynomial $f$), and as above we deduce that $V_{ss}=V_1$ exactly when it does not vanish, and this space has dimension $np$. Thus the first vertex of the generic Newton polygon is the point $(n(p-1)p,p)$, with Hasse polynomial the above determinant.

The case $d=p^{n+1}-t$, $n\geq 2$, $p+1\leq t\leq p^2-1$, follows exactly the same lines, and we get the same result, the only difference being that the indices in the sum defining the $\theta_i$ run over $\{t,\ldots,p^2-1\}$.

\subsubsection{$ p^2-2p+2\leq d\leq p^2-p-1$} The case $d=p^2-p-1$ comes in the same way from Proposition \ref{minirr2}. We get $((p+1)(p-1),p+1)$ as first vertex, and the Hasse polynomial has the same form as above, with $\theta_j=\frac{1}{2}\sum_{i=p+1}^{2p-2} c_{p^{2}-i}^{p}\left\{f^2\right\}_{pi-j}$, $1\leq j\leq p-1$.

The case $d=p^2-t$, $p+1\leq t\leq 2p-2$ gives the same result, the only difference being that the indices in the sum defining the $\theta_i$ run over $\{t,\ldots,2p-2\}$.

\subsubsection{$p^{n+1}-p+1\leq d\leq p^{n+1}-2$, $n\geq 1$} We finally treat the case $d=p^{n+1}-2$, with $n\geq 1$. Here the matrix $\overline{M}(\Gamma)$ is no longer invertible. From Proposition \ref{cas1} (i), the minimal support is
$$\{1,\ldots,p^n,ip^k,~ 2\leq i\leq p-1,~ 0\leq k\leq n-1\}$$
If we denote by $e(i)$ the basis vector of $V$ associated to the element $i$ of the minimal support, we can describe the action of $\varphi$ (the semi linear morphism of $V$ whose matrix is the transpose of $\overline{M}(\Gamma)$) on this basis: we get
$$\varphi(e(ip^k))=e(ip^{k+1}),~ 0\leq k\leq n-2,~ 1\leq i\leq p-1,~ \varphi(e(p^{n-1}))=e(p^n)$$
$$\varphi(ip^{n-1})=c_{ip^n-1}e(1),~ 2\leq i\leq p-1,~ f(e(p^n))=\sum_{i=2}^{p-1} c_{p^{n+1}-i}e(i)$$
We see that for any basis vector $e(i)$, one of its iterates lands in the space $V_1$ generated by $e(1)$ and its iterates; we deduce that $V_{ss}$ is contained in $V_1$. Assuming that at least one of the products $c_{p^{n+1}-i}c_{ip^n-1}$, $2\leq i\leq p-1$ is non zero, the $2n+1$ vectors $e(1),\varphi(e(1)),\ldots,\varphi^{2n}(e(1))$ are linearly independent, and we have
$$\varphi^{2n+1}(e(1))=\left(\sum_{i=2}^{p-1} c_{p^{n+1}-i}^{p^n}c_{ip^n-1}\right)e(1)$$
We deduce that the restriction of $\varphi$ to $V_1$ is an isomorphism if, and only if the sum above is non zero; in this case we must have $V_1=V_{ss}$, and $V_{ss}$ has dimension $2n+1$. In other words, the first vertex of the generic Newton polygon $\GNP(d,p)$ is $((2n+1)(p-1),2)$, with associated Hasse polynomial the above sum.

In the case $d=p^{n+1}-t$, $2\leq t\leq p-1$, we reason the same way, except that the minimal support is smaller since the $i$ in $ip^k$ runs over $\{t,\ldots,p-1\}$, and the sum defining the Hasse polynomial runs over this last index set.


\end{document}